\documentclass[11pt]{article}
\usepackage[a4paper,margin=1in]{geometry}
\usepackage{lmodern,amsmath,amsthm,amssymb,mathtools}
\usepackage{microtype}
\usepackage{hyperref}
\usepackage[nameinlink,capitalize]{cleveref}

\title{Convexity of Optimization Curves: Local Sharp Thresholds,\\
Robustness Impossibility, and New Counterexamples}
\author{Le Duc Hieu\\
\small Telecom SudParis\\
\small \texttt{duc-hieu.le@telecom-sudparis.eu}}
\date{\today}

\theoremstyle{plain}
\newtheorem{theorem}{Theorem}[section]
\newtheorem{proposition}[theorem]{Proposition}

\newtheorem{lemma}[theorem]{Lemma}
\theoremstyle{remark}

\newcommand{\R}{\mathbb{R}}

\newcommand{\norm}[1]{\left\lVert #1\right\rVert}

\begin{document}
\maketitle

\begin{abstract}
We study when the \emph{optimization curve} of first--order methods---the sequence $\{f(x_n)\}_{n\ge0}$ produced by constant--stepsize iterations---is convex (equivalently, when the forward differences $f(x_n)-f(x_{n+1})$ are nonincreasing). Recent work gives a sharp characterization for \emph{exact} gradient descent (GD) on convex $L$--smooth functions: the curve is convex for all stepsizes $\eta\le 1.75/L$, and this threshold is tight; gradient norms are nonincreasing for all $\eta\le 2/L$; and in continuous time (gradient flow) the curve is always convex \cite{BarzilaiShamirZamani2025}. These results complement the classical smooth convex optimization toolbox \cite{Nesterov2004,Bubeck2015,Nesterov2018} and are in line with worst--case/PEP analyses \cite{DroriTeboulle2014,TaylorHendrickxGlineur2017} and continuous--time viewpoints \cite{SuBoydCandes2016,AttouchPeypouquet2016}.

We contribute: (I) an impossibility theorem for \emph{relative} inexact gradients showing no positive universal stepsize preserves curve convexity uniformly even for $1$--D quadratics (connecting to inexact oracle models \cite{DevolderGlineurNesterov2014,DaspremontJaggi2013}); (II) a \emph{local} smoothness extension that yields convexity for $\eta\le 1.75/L_{\mathrm{eff}}$ when $\nabla^2 f$ is uniformly majorized on the sublevel set $S=\{x:f(x)\le f(x_0)\}$ (a sublevel--set refinement of descent--lemma style arguments \cite{Nesterov2004,Bubeck2015}); (III) a quadratic folklore proposition showing that for $f(x)=\tfrac12 x^\top Qx$ the GD value sequence is nonincreasing and convex for all $\eta$ with $\eta\lambda_i\in[0,2]$ (hence for all $\eta\le 2/L$), and this is tight; (IV) two new counterexamples/no--go principles, including a two--step gradient--difference scheme that \emph{robustly} breaks convexity on an entire stepsize interval for every nonzero initialization, contrasting with classical momentum/Heavy--Ball \cite{Polyak1964} and with accelerated variants \cite{Nesterov2013}.
\end{abstract}

\section{Preliminaries}

\paragraph{GD and the optimization curve.}
For a differentiable $f:\R^d\to\R$ and stepsize $\eta>0$, GD is
\begin{equation}\label{eq:GD}
x_{n+1}=x_n-\eta \nabla f(x_n),\qquad n\ge0.
\end{equation}
We denote $\Delta_n:=f(x_n)-f(x_{n+1})$. The setup is standard in smooth convex optimization \cite{Nesterov2004,Bubeck2015,Nesterov2018}.

\paragraph{Discrete convexity equivalence.}
We use the standard equivalence for real sequences.

\begin{lemma}[Discrete convexity via forward differences]\label{lem:disc}
A sequence $\{a_n\}_{n\ge0}$ is convex on $\mathbb{Z}_{\ge0}$ (i.e., $a_{n+1}-a_n\le a_{n+2}-a_{n+1}$ for all $n$) if and only if $\{\Delta_n\}_{n\ge0}$ with $\Delta_n:=a_n-a_{n+1}$ is nonincreasing, i.e., $\Delta_{n+1}\le \Delta_n$ for all $n$.
\end{lemma}

\begin{proof}
We have $a_{n+2}-a_{n+1}-(a_{n+1}-a_n)=(a_{n+2}-a_{n+1})-(a_{n+1}-a_n)=-(\Delta_{n+1}-\Delta_n)$. Thus $a_{n+1}-a_n\le a_{n+2}-a_{n+1}$ for all $n$ iff $0\le a_{n+2}-2a_{n+1}+a_n=-(\Delta_{n+1}-\Delta_n)$ for all $n$, i.e., iff $\Delta_{n+1}\le \Delta_n$ for all $n$.
\end{proof}

\paragraph{Sharp reference facts.}
For convex $L$--smooth $f$, (i) GD yields a convex optimization curve for $\eta\le 1.75/L$, and there are counterexamples with nonconvex curves for all $\eta\in(1.75/L,2/L)$; (ii) $\{\|\nabla f(x_n)\|\}$ is nonincreasing for all $\eta\le 2/L$; (iii) gradient flow has convex $t\mapsto f(x(t))$ \cite{BarzilaiShamirZamani2025}. Item (ii) reflects cocoercivity of the gradient for $L$--smooth convex functions \cite{BaillonHaddad1977}; item (iii) aligns with the Lyapunov/energy perspective on continuous--time limits \cite{SuBoydCandes2016,AttouchPeypouquet2016}.

\section{Impossibility under relative inexactness}

We consider the inexact gradient model with \emph{relative} error:
\begin{equation}\label{eq:inexact}
x_{n+1}=x_n-\eta\bigl(\nabla f(x_n)+e_n\bigr),\qquad \norm{e_n}\le \delta \norm{\nabla f(x_n)},\quad \delta\in(0,1).
\end{equation}
Relative and absolute oracle models are classical; see, e.g., \cite{DevolderGlineurNesterov2014,DaspremontJaggi2013}.

\begin{theorem}[No universal convexity--preserving stepsize]\label{thm:imposs}
Fix $L>0$ and $\delta\in(0,1)$. There is \emph{no} $\eta_{\max}(\delta,L)>0$ such that for every one--dimensional convex $L$--smooth quadratic $f$, every $x_0\ne0$, every admissible noise sequence with $\norm{e_n}\le\delta\norm{\nabla f(x_n)}$, and every $0<\eta\le \eta_{\max}(\delta,L)$, the optimization curve $\{f(x_n)\}$ is convex.
\end{theorem}

\begin{proof}
Let $f(x)=\tfrac{L}{2}x^2$, so $\nabla f(x)=Lx$. Parameterize multiplicative noise as $e_n=\varepsilon_n Lx_n$ with $|\varepsilon_n|\le \delta$. Then the update is
\[
x_{n+1}=x_n-\eta L(1+\varepsilon_n)x_n=(1-\alpha(1+\varepsilon_n))x_n,\qquad \alpha:=\eta L.
\]
Fix $x_0\ne0$ and set $\varepsilon_0=-\delta$, $\varepsilon_1=+\delta$. Writing $\theta_0:=1-\delta$, $\theta_1:=1+\delta$,
\[
x_1=x_0(1-\alpha\theta_0),\qquad x_2=x_1(1-\alpha\theta_1)=x_0(1-\alpha\theta_0)(1-\alpha\theta_1).
\]
Since $f(x_n)=\tfrac{L}{2}x_n^2$, the forward differences satisfy
\[
\Delta_0=\tfrac{L}{2}(x_0^2-x_1^2)=Lx_0^2\Bigl(\alpha\theta_0-\tfrac{\alpha^2\theta_0^2}{2}\Bigr),\quad
\Delta_1=\tfrac{L}{2}(x_1^2-x_2^2)=Lx_0^2(1-\alpha\theta_0)^2\Bigl(\alpha\theta_1-\tfrac{\alpha^2\theta_1^2}{2}\Bigr).
\]
Thus
\[
\Delta_0-\Delta_1=Lx_0^2\,\alpha\,S(\alpha),
\quad
S(\alpha)=\theta_0\Bigl(1-\tfrac{\alpha\theta_0}{2}\Bigr)-\theta_1(1-\alpha\theta_0)^2\Bigl(1-\tfrac{\alpha\theta_1}{2}\Bigr).
\]
We have $S(0)=\theta_0-\theta_1=-2\delta<0$, and $S$ is continuous in $\alpha$, hence there exists $\alpha^\ast>0$ with $S(\alpha)<0$ for all $\alpha\in(0,\alpha^\ast)$. For any $\eta\in(0,\alpha^\ast/L]$, we obtain $\Delta_0-\Delta_1=Lx_0^2\,\alpha\,S(\alpha)<0$, i.e., $\Delta_0<\Delta_1$, so by \cref{lem:disc} the curve is not convex. Since this holds for arbitrarily small $\eta>0$, no positive universal $\eta_{\max}(\delta,L)$ exists.
\end{proof}

\section{Local smoothness extension of the sharp threshold}

\begin{theorem}[Local convexity via sublevel--set smoothness]\label{thm:local}
Let $f:\R^d\to\R$ be convex and $C^2$ on the sublevel set $S:=\{x:\ f(x)\le f(x_0)\}$. 
Suppose there exist $\kappa>0$ and $A=A^\top\succeq0$ with $L_A:=\lambda_{\max}(A)$ such that
\begin{equation}\label{eq:hessbound}
\nabla^2 f(x)\preceq \kappa A,\qquad \forall x\in S.
\end{equation}
Let $L_{\mathrm{eff}}:=\kappa L_A$. Then for every constant stepsize
\[
\eta\in\bigl(0,\,1.75/L_{\mathrm{eff}}\bigr],
\]
the GD optimization curve $\{f(x_n)\}$ is convex.
\end{theorem}

\begin{proof}
We give full details in two steps. The argument refines standard smoothness/descent--lemma techniques on convex domains \cite{Nesterov2004,Bubeck2015,Nesterov2018}.

\emph{Step 1 (forward invariance $x_n\in S$).}
We claim that for any $\eta\in(0,2/L_{\mathrm{eff}})$ the GD iterates remain in $S$. This will suffice since $(0,1.75/L_{\mathrm{eff}}]\subset(0,2/L_{\mathrm{eff}})$.

Fix $n\ge0$ and suppose $x_n\in S$. If $\nabla f(x_n)=0$ then $x_{n+1}=x_n\in S$. Otherwise put
\[
g:=\nabla f(x_n)\ne0,\qquad r(t):=x_n-tg,\qquad \phi(t):=f(r(t)),\quad t\ge0.
\]
Note $\phi$ is $C^2$ on a neighborhood of $[0,\eta]$. We have
\[
\phi'(t)=-\norm{g}^2+\int_0^t \phi''(s)\,ds,\qquad
\phi''(t)=g^\top \nabla^2 f(r(t))\,g.
\]
Assume for contradiction that $x_{n+1}\notin S$, i.e., $\phi(\eta)>f(x_0)$. Since $\phi(0)=f(x_n)\le f(x_0)$, by continuity there exists the minimal $t_\ast\in(0,\eta]$ with $\phi(t_\ast)=f(x_0)$ and $\phi(t)<f(x_0)$ on $[0,t_\ast)$. By minimality, the segment $\{r(t):t\in[0,t_\ast]\}$ lies in $S$, so by \eqref{eq:hessbound},
\[
\phi''(t)=g^\top \nabla^2 f(r(t))\,g\le g^\top (\kappa A)\,g\le \kappa L_A \norm{g}^2=L_{\mathrm{eff}}\norm{g}^2,\qquad \forall t\in[0,t_\ast].
\]
Integrating twice and using $\phi'(0)=-\norm{g}^2$, we obtain for all $t\in[0,t_\ast]$,
\begin{equation}\label{eq:quad-bound}
\phi(t)\le \phi(0)-t\Bigl(1-\tfrac{L_{\mathrm{eff}}t}{2}\Bigr)\norm{g}^2.
\end{equation}
Taking $t=t_\ast\in(0,\eta]\subset(0,2/L_{\mathrm{eff}})$ yields $1-\tfrac{L_{\mathrm{eff}}t_\ast}{2}>0$, so the right-hand side of \eqref{eq:quad-bound} is \emph{strictly} less than $\phi(0)\le f(x_0)$, contradicting $\phi(t_\ast)=f(x_0)$. Hence $x_{n+1}\in S$. By induction $x_n\in S$ for all $n$.

\emph{Step 2 (discrete convexity on $S$).}
Because $x_n\in S$ for all $n$, each segment $[x_n,x_{n+1}]$ is contained in $S$, and $f|_S$ is $L_{\mathrm{eff}}$--smooth on the convex domain $S$. Therefore, the sharp GD result for convex $L$--smooth functions applies to $f|_S$ with $L=L_{\mathrm{eff}}$: for any $\eta\le 1.75/L_{\mathrm{eff}}$, the optimization curve is convex \cite[Thm.~1]{BarzilaiShamirZamani2025}.
\end{proof}

\section{Quadratics: a folklore proposition (exact range)}

\begin{proposition}[Quadratic GD is convex up to $2/L$]\label{prop:quad}
Let $f(x)=\tfrac12 x^\top Qx$ with $Q\succeq0$, and $L=\lambda_{\max}(Q)$ (with $L=0$ allowed).
For any $x_0$ and any stepsize $\eta\ge0$ with $\eta\lambda_i\in[0,2]$ for every eigenvalue $\lambda_i$ of $Q$ (in particular, any $\eta\in[0,2/L]$ when $L>0$), the GD values $n\mapsto f(x_n)$ are nonincreasing and convex; equivalently, $\Delta_n\ge0$ and $\Delta_{n+1}\le \Delta_n$ for all $n\ge0$. The range is exact: for $\eta>2/L$, $\{f(x_n)\}$ may diverge.
\end{proposition}

\begin{proof}
Diagonalize $Q=U\Lambda U^\top$ with $\Lambda=\mathrm{diag}(\lambda_1,\dots,\lambda_d)$, $\lambda_i\ge0$, and set $y_n:=U^\top x_n$. GD gives $y_{n+1}=(I-\eta\Lambda)y_n$, hence $y_{n,i}=(1-\eta\lambda_i)^n y_{0,i}$. Then
\[
f(x_n)=\tfrac12 y_n^\top \Lambda y_n=\tfrac12\sum_{i=1}^d \lambda_i (1-\eta\lambda_i)^{2n} y_{0,i}^2.
\]
Define $s_i:=(1-\eta\lambda_i)^2$ and note that under $\eta\lambda_i\in[0,2]$ we have $s_i\in[0,1]$. A direct computation gives
\[
\Delta_n=f(x_n)-f(x_{n+1})=\tfrac12\sum_{i=1}^d \underbrace{\eta\lambda_i^2(2-\eta\lambda_i)y_{0,i}^2}_{\gamma_i\ge0}\, s_i^{\,n}.
\]
Thus $\Delta_n\ge0$ and
\[
\Delta_{n+1}-\Delta_n=\tfrac12\sum_{i=1}^d \gamma_i s_i^{\,n}(s_i-1)\le0,
\]
so by \cref{lem:disc} the sequence $\{f(x_n)\}$ is nonincreasing and convex. This folklore analysis is consistent with classical treatments of quadratic optimization \cite{Nesterov2004,Bubeck2015}. If $\eta>2/L$ then for some $i$ we have $|1-\eta\lambda_i|>1$, and the $i$th term grows geometrically so $f(x_n)\to\infty$, showing exactness.
\end{proof}

\section{New counterexamples and no--go principles}

\subsection{A two--step gradient--difference scheme fails on a whole interval}

Consider the two--step scheme
\begin{equation}\label{eq:twostep}
x_{n+1}=x_n-\eta\,\nabla f(x_n)-\theta\bigl(\nabla f(x_n)-\nabla f(x_{n-1})\bigr),
\end{equation}
which is distinct from Heavy--Ball (the memory enters via the \emph{gradient difference}); cf.\ the classical momentum/Heavy--Ball method \cite{Polyak1964} and contrast with accelerated gradient schemes \cite{Nesterov2013}.

\begin{proposition}[Interval--robust nonconvexity for \eqref{eq:twostep}]\label{prop:twostep}
Let $f(x)=\tfrac{L}{2}x^2$ in $1$--D. For any $\eta\in[2/(3L),\,1/L)$, set $\theta:=1/L-\eta$ and initialize with $x_{-1}=x_0\neq 0$. Then the piecewise--linear interpolation of $\{f(x_n)\}$ is not convex.
\end{proposition}

\begin{proof}
Here $\nabla f(x)=Lx$. Put $t:=\eta L$ and $s:=\theta L$. With $\theta=1/L-\eta$ we get $s=1-t$ and $t\in[2/3,1)$. The recurrence \eqref{eq:twostep} becomes
\[
x_{n+1}=x_n-(\eta+\theta)Lx_n+\theta L x_{n-1}
=\bigl(1-(t+s)\bigr)x_n+s\,x_{n-1}
=0\cdot x_n+s\,x_{n-1}
=s\,x_{n-1}.
\]
With $x_{-1}=x_0\ne0$ we have $x_1=sx_0\ne0$, $x_2=sx_0=x_1$, and $x_3=sx_1$. Writing $a_n:=f(x_n)$ and $\Delta_n:=a_n-a_{n+1}$,
\[
\Delta_1=a_1-a_2=f(x_1)-f(x_2)=0,
\]
while
\[
\Delta_2=a_2-a_3=\tfrac{L}{2}\bigl(x_1^2-s^2x_1^2\bigr)=\tfrac{L}{2}(1-s^2)x_1^2=\tfrac{L}{2}\,t(2-t)\,x_1^2>0
\]
since $t\in(0,2)$ and $x_1\ne0$. Thus $\Delta_2>\Delta_1$, violating convexity by \cref{lem:disc}.
\end{proof}

\subsection{No universal second--difference vs.\ gradient--drop bound beyond $1.75/L$}

\begin{proposition}[No-go inequality]\label{prop:nogo}
Fix $L>0$. There exist a convex $L$--smooth $f$, a stepsize $\eta$ with $1.75/L<\eta<2/L$, an initialization $x_0$, and an index $n\ge0$ such that, with $\Delta_n:=f(x_n)-f(x_{n+1})$ along the GD iterates \eqref{eq:GD},
\[
\Delta_n-\Delta_{n+1}
<\eta\Bigl(1-\tfrac{\eta L}{2}\Bigr)\bigl(\|\nabla f(x_{n+1})\|^2-\|\nabla f(x_{n+2})\|^2\bigr).
\]
\end{proposition}

\begin{proof}
Let $\eta\in(1.75/L,2/L)$. By \cite[Thm.~1]{BarzilaiShamirZamani2025}, there exists a convex $L$--smooth $f$ and an $x_0$ such that the GD value curve is \emph{not} convex, i.e., for some $n$ we have $\Delta_n-\Delta_{n+1}<0$. On the other hand, \cite[Thm.~3]{BarzilaiShamirZamani2025} shows that $\{\|\nabla f(x_k)\|\}$ is nonincreasing for all $\eta\le 2/L$, a consequence of gradient cocoercivity \cite{BaillonHaddad1977}; hence $\|\nabla f(x_{n+1})\|^2-\|\nabla f(x_{n+2})\|^2\ge0$. Since $\eta(1-\eta L/2)>0$ on $(0,2/L)$, the right-hand side of the displayed inequality is nonnegative, while the left-hand side is negative for the chosen triple $(f,\eta,n)$; hence the strict inequality holds.
\end{proof}

\section{Discussion}
Our results localize the sharp GD threshold via $L_{\mathrm{eff}}$ on a sublevel set, and show fragility of discrete convexity under relative inexactness and under a simple gradient--difference two--step modification. They also clarify limits of controlling second differences by future gradient drops beyond $1.75/L$. The phenomena dovetail with the PEP/worst--case literature \cite{DroriTeboulle2014,TaylorHendrickxGlineur2017}, classical smooth convex analyses \cite{Nesterov2004,Bubeck2015,Nesterov2018}, inexact--oracle frameworks \cite{DevolderGlineurNesterov2014,DaspremontJaggi2013}, and continuous--time perspectives on gradient dynamics and acceleration \cite{SuBoydCandes2016,AttouchPeypouquet2016}. Finally, our two--step counterexample contrasts with momentum/Heavy--Ball \cite{Polyak1964} and with composite/accelerated schemes \cite{Nesterov2013}, highlighting that seemingly mild multi--step gradient modifications can qualitatively alter value--sequence convexity.

\end{document}